\documentclass[11pt]{article}
\usepackage{amsfonts,amsmath,amssymb}
\usepackage{hyperref,amsthm}
\usepackage{epsfig}
\usepackage{graphicx}

\newtheorem{theorem}{Theorem}
\newtheorem{proposition}[theorem]{Proposition}
\newtheorem{lemma}[theorem]{Lemma}
\newtheorem{claim}[theorem]{Claim}

\newtheorem{notation}[theorem]{Notation}
\newtheorem{question}[theorem]{Question}
\newtheorem{observation}[theorem]{Observation}
\newtheorem{definition}[theorem]{Definition}

\begin{document}

\title{On Convex Geometric Graphs with no $k+1$
Pairwise Disjoint Edges}

\author{
Chaya Keller and Micha A. Perles \\
Einstein Institute of Mathematics, Hebrew University\\
Jerusalem 91904, Israel\\
}

\maketitle

\begin{abstract}

A well-known result of Kupitz from 1982 asserts that the maximal
number of edges in a convex geometric graph (CGG) on $n$ vertices
that does not contain $k+1$ pairwise disjoint edges is $kn$ (provided
$n>2k$). For $k=1$ and $k=n/2-1$, the extremal examples are completely
characterized. For all other values of $k$, the structure of the extremal
examples is far from known: their total number is unknown, and only a few
classes of examples were presented, that are almost symmetric,
consisting roughly of the $kn$ ``longest possible'' edges of $CK(n)$, the
complete CGG of order $n$.

In order to understand further the structure of the extremal examples, we
present a class of extremal examples that lie at the other end of the
spectrum. Namely, we break the symmetry by requiring that, in addition,
the graph admit an independent set that consists of $q$ consecutive
vertices on the boundary of the convex hull. We show that such graphs
exist as long as $q \leq n-2k$ and that this value of $q$ is optimal.

We generalize our discussion to the following question: what is the
maximal possible number $f(n,k,q)$ of edges in a CGG on $n$ vertices
that does not contain $k+1$ pairwise disjoint edges, and, in addition,
admits an independent set that consists of $q$ consecutive vertices
on the boundary of the convex hull? We provide a complete answer to
this question, determining $f(n,k,q)$ for all relevant values of $n,k$
and $q$.
\end{abstract}

\section{Introduction}

A {\it geometric graph} (GG) is a graph whose vertices are points in
general position in the plane, and whose edges are segments
connecting pairs of vertices. If the vertices are in convex position
(i.e., no vertex lies in the convex hull of the remaining vertices),
the graph is called a {\it convex geometric graph} (CGG).


One of the oldest Tur\'{a}n-type questions for geometric graphs, mentioned already
by Erd\H{o}s~\cite{Erdos46} in 1946, is: what is the maximal number of
edges in a geometric graph on $n$ vertices that does not contain $k+1$
{\it pairwise disjoint} edges? This question, along with the ``dual'' question
concerning geometric graphs that do not contain $k+1$ {\it pairwise crossing}
edges, became central research topics in geometric combinatorics, with
dozens of papers obtaining upper and lower bounds and resolving special
cases (see, e.g.,~\cite{ACFFHHHW10,AA89,CP92,Cerny05,DKM02,Felsner,FPS13}).

For the sake of brevity, we call a GG that does not contain $k+1$ pairwise disjoint edges
\emph{an $I_{k+1}$-free graph}, and a GG that does not contain $k+1$ pairwise crossing edges
\emph{an $X_{k+1}$-free graph}.

For general GGs, the questions of determining the maximal possible number of edges in an $I_{k+1}$-free
graph and in an $X_{k+1}$-free graph are still widely open. For $I_{k+1}$-free GGs, the best
currently known upper bound is $2^8 k^2 n$, obtained by Felsner~\cite{Felsner}, and the
best lower bound is $\frac{3}{2}(k-1)n-2k^2$, obtained by T\'{o}th and Valtr~\cite{TV99}.
For $X_{k+1}$-free GGs, the best known upper bound is $2^{ck^6} n \log n$, obtained by
Fox, Pach, and Suk~\cite{FPS13}. For both questions, it is conjectured that the
correct answer is of order $\Theta(kn)$ (see~\cite{Pach15}).

Much more is known in the convex case. For $X_{k+1}$-free CGGs, Capoyleas and
Pach~\cite{CP92} showed in 1992 that the maximal possible number of edges is
$2kn - {{2k+1}\choose{2}}$, for any $n \geq 2k+1$. It was also shown
that any $X_{k+1}$-free CGG that is maximal with respect to inclusion (i.e., addition
of any edge gives rise to $k+1$ pairwise crossing edges) contains exactly
$2kn - {{2k+1}\choose{2}}$ edges~\cite{DKM02}, and the exact number of extremal examples was
shown to be equal to the determinant of a certain matrix of Catalan numbers~\cite{Jonsson05}.
Strong upper bounds were obtained also for
related questions, such as determining the maximal possible number of edges in a CGGs that does
not contain a $k$-grid (i.e., two families of $k$ edges each such that each edge in the first
family crosses each edge in the other family), see~\cite{AFPS14,BMP}.

For $I_{k+1}$-free CGGs, Kupitz~\cite{Kupitz} showed in 1982 that the maximal possible number
of edges is $kn$ (for $n \geq 2k+1$) or ${{n}\choose{2}}$ (for $0 \leq n \leq 2k+1$). In the
cases $k=1$ (for all $n$) and $k=n/2-1$ (for even $n$), the exact
characterization of all \emph{maximal} $I_{k+1}$-free CGGs (i.e., all $I_{k+1}$-free CGGs with
exactly $kn$ edges) is known. For $k=1$, Woodall~\cite{Woodall} showed that the maximal CGGs are all
self-intersecting circuits of odd order with inward pointing leaves. This characterization holds
also for general (i.e., not necessarily convex) geometric graphs, and is
related to Conway's {\it thrackle conjecture} (see~\cite{BMP}, p.~306).
For $k=n/2-1$ , a characterization was given recently in~\cite{Keller-Perles}. The extremal graphs
turn out to be complements of trees of a special structure called {\it caterpillar trees}
(see~\cite{Caterpillar1}).

As for characterization of the maximal $I_{k+1}$-free CGGs for $1< k < n/2-1$, the knowledge is very
scarce. Unlike the case of maximal $X_{k+1}$-free CGGs, the number of maximal $I_{k+1}$-free CGGs is
not known, and only a few examples of such CGGs were given (see, e.g.,~\cite{Felsner}).
Moreover, all known examples are of a very specific kind, containing roughly the
$kn$ edges of maximal order in $CK(n)$, the complete CGG on $n$ vertices (see Section~\ref{sec:notations}
for the formal definition of ``order''). If $CK(n)$ is represented by a regular polygon $P$ of order $n$
with all its diagonals, such graphs are almost fully symmetric about
the center of the polygon, and they do not contain an independent set of more than $\lceil n/2 \rceil -k$
consecutive vertices on the boundary of $P$.

\medskip

In this paper, we seek to broaden our acquaintance with the family of maximal $I_{k+1}$-free CGGs.
To this end, we consider examples that are as far as possible from being centrally
symmetric. Formally, we ask how large can be $q$, such that there exists a maximal $I_{k+1}$-free
graph that contains an independent set of $q$ consecutive vertices on the boundary of $P$. Such
a restriction forces the graph to be far from symmetric, as demonstrated in
Figure~\ref{Fig:1New}. We show by an explicit class of examples that $q$ can be as large as
$n-2k$ (i.e., twice as large as the value of $q$ for the previously known examples), and
prove that this value is maximal.

We obtain this result as a special case of a more general result concerning $I_{k+1}$-free
graphs with a large independent set of consecutive boundary vertices.

\begin{definition}
Let $G$ be a CGG. A set $A \subset V(G)$ is a \emph{free boundary arc
of order $q$} if:
\begin{itemize}
\item $A$ consists of $q$ consecutive vertices on the boundary of $\mathrm{conv} (V(G))$,

\item $A$ is independent, i.e., no edge of $G$ connects two vertices of $A$.
\end{itemize}
\end{definition}

\noindent We often say that $G$ \emph{avoids} $A$, when $A$ is an independent set of
vertices of $G$.

The general question we address can be formulated as follows:
\begin{question}
What is the maximal number of edges in an $I_{k+1}$-free CGG on $n$ vertices that
admits a free boundary arc of order $q$?
\end{question}

We provide a complete answer in the following theorem:
\begin{theorem}\label{Thm:Main}
Let $n,k,$ and $q$ be integers such that $1 \leq k \leq \lfloor n/2 \rfloor -1$
and $1 \leq q \leq n-1$.
Denote by $f(n,k,q)$ the maximal number of edges in an $I_{k+1}$-free CGG on $n$
vertices that admits a free boundary arc of order $q$. Then:
\begin{enumerate}
\item If $q \leq n-2k$, then $f(n,k,q)=kn$. In particular, there exists a \emph{maximal}
$I_{k+1}$-free CGG that admits a free boundary arc of order $q$.

\item If $q=n-2k+ \ell$ for $0 < \ell < k$, then $f(n,k,q)=
kn- {{\ell+1}\choose{2}}$.

\item If $q \geq n-k$, then $f(n,k,q)= {{n}\choose{2}} -
{{q}\choose{2}}$.
\end{enumerate}
\end{theorem}

The proof that the stated values are upper bounds for $f(n,k,q)$ is quite straightforward. The
proof of the equalities is somewhat more complex, and involves a construction of an asymmetric
maximal $I_{k+1}$-free CGG. To illustrate the construction, two examples of CGGs with $kn$ edges
that admit a free boundary arc of order $n-2k$ are given in Figure~\ref{Fig:1New}.

\begin{figure}[tb]
\begin{center}
\scalebox{1.0}{
\includegraphics{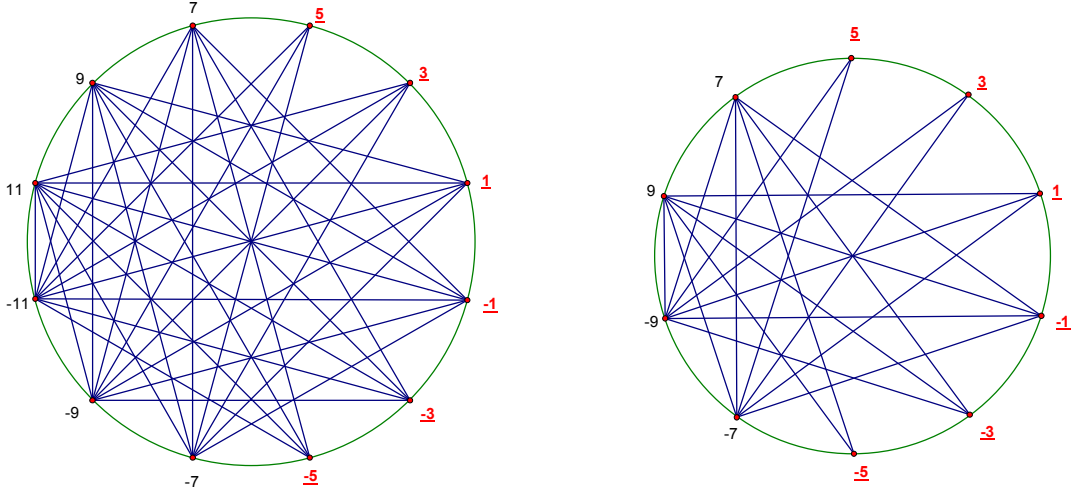}
}
\caption{Two maximal CGGs that admit a free boundary arc of order $n-2k$. On the left,
$n=12$ and $k=3$, and on the right, $n=10$ and $k=2$. The avoided vertices are
colored red and underlined.}
\label{Fig:1New}
\end{center}
\end{figure}

This paper is organized as follows: In
Section~\ref{sec:notations} we introduce a few definitions and
notations, and prove the upper bound of Theorem~\ref{Thm:Main}. In
Section~\ref{sec:basic-construction} we present a family
of maximal $I_{k+1}$-free graphs $G_{n,k}$ that avoid $q=n-2k$
consecutive vertices, thus proving Theorem~\ref{Thm:Main}(1). In
Section~\ref{sec:extension} we modify the construction to produce $I_{k+1}$-free graphs
$G_{n,k,\ell}$ that avoid $q=n-2k+\ell$
consecutive vertices, and use
this construction to complete the proof of Theorem~\ref{Thm:Main}.







\section{Preliminaries}
\label{sec:notations}

We start with a few definitions and notations, which are a bit unusual but
will be very convenient for the proofs in the sequel.


\subsection{Definitions and Notations}
\label{sec:sub:notations}

Let $n$ be a positive integer, $n \geq 4$. In order to deal with a convex geometric
graph $G=(V(G),E(G))$ of order $n$ with a free arc $A$, we consider a regular
polygon $P$ of order $2n$, whose vertices lie on a circle with center $O$ and are
numbered cyclically from $-n+1$ to $n$, as shown in Figure~\ref{Fig:2New}. We shall
always insist that the free arc $A$ \textbf{is} symmetric with respect to the horizontal
axis through $O$ and lie on the right. Thus, if $|A|$ is even, then we let
$V(G)$ be the set of odd-labelled vertices of $P$, and if $|A|$ is odd, we let $V(G)$
be the set of even-labelled vertices of $P$. The segment connecting the endpoints of
$A$ will always be vertical.

\begin{figure}[tb]
\begin{center}
\scalebox{1.0}{
\includegraphics{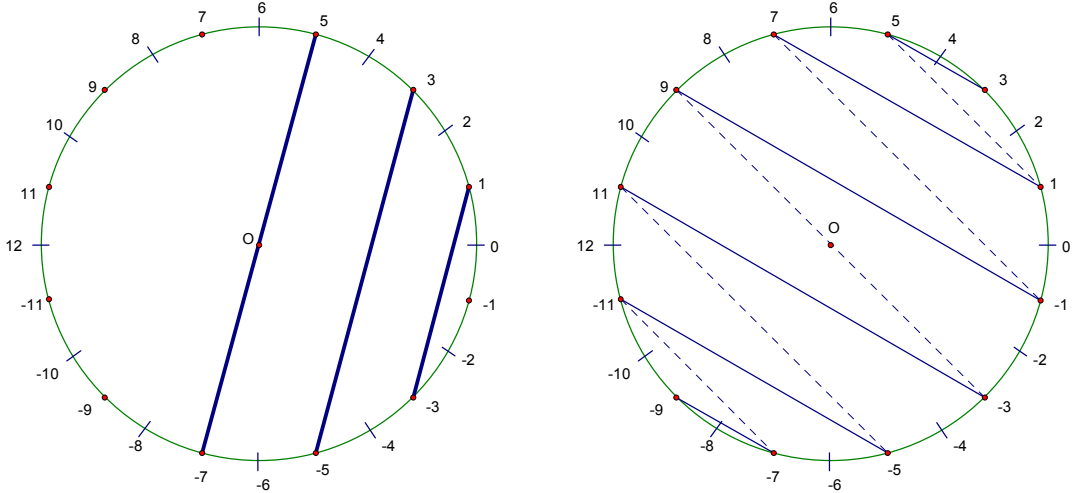}
} \caption{Representation of a CGG $G$ on $n=12$ vertices using a regular polygon $P$
on $24$ vertices, in the case of even $|A|$. The vertices of $G$ are the odd-labelled vertices of $P$. On
the right, all possible edges of direction $4$ are drawn as ordinary
lines and all possible edges of direction $3$ are drawn as dotted lines. On the left,
the set of bold edges is $B_{-1,5}$ (assuming $k=3$).}
\label{Fig:2New}
\end{center}
\end{figure}

Given the set $A$ avoided by $G$, we say that an edge $e$ is {\it allowed} if at
least one of its endpoints lies in $A^c$ (i.e., $V(G) \setminus A$). Otherwise,
$e$ is called {\it forbidden}. We also call the vertices of $A$ {\it avoided} or
{\it forbidden} vertices, and the vertices of $A^c$ {\it allowed} vertices.

We define the {\it order} of an edge $[i,j] \in E(G)$ to be
$ord([i,j])= \frac{1}{2} \min(|i-j|, 2n - |i-j|)$. Note that each edge $e \in E(G)$
divides the boundary of $P$ into two open arcs. If $ord(e)=r$, then these two
arcs contain $r-1$ and $n-r-1$ vertices of $G$, respectively.

We say that an edge $e = [i,j] \in E(G)$ {\it emanates
from} its left vertex, according to the labelling in
Figure~\ref{Fig:2New}. Formally, $e$ emanates from $i$ if
$|i| > |j|$. In the case $i=-j$, we say that the edge $e$
emanates from the positive-labelled vertex.

For each vertex $v$ of $P$, we consider the set of edges on $V(G)$ that
are perpendicular to the radius $\overrightarrow{Ov}$. The edges of this
set are parallel. If $n$ is odd, then their number is $(n-1)/2$. If $n$ is
even, then their number is $n/2$ for $v \in \mathrm{vert}(P) \setminus V(G)$
and $n/2-1$ for $v \in V(G)$. We say that these are the edges {\it in
direction $v$.} Note that in this notation, the vertices labelled
$i$ and $i-n$ (for any $1 \leq i \leq n$) correspond to the same direction.
See the right part of Figure~\ref{Fig:2New} for an illustration of the
directions in a CGG on $12$ vertices.

A pair of parallel edges on $V(G)$ is called {\it consecutive} if each vertex
of one edge is adjacent on the boundary of $V(G)$ to a vertex of the other edge.
Throughout the paper, we make use of sequences of
consecutive parallel edges (i.e., sequences of the form $\langle e_1,e_2,\ldots,e_{\ell}
\rangle$, where $e_i,e_{i+1}$ are consecutive for each $1 \leq i \leq \ell-1$). Of course,
all edges in such a sequence belong to the same direction.
For each such sequence $S$ belonging to the direction $i$, the endpoints of the extremal
edges of $S$ determine two disjoint open arcs on the boundary of $P$ that contain all
the remaining vertices of $G$.
We denote the arc that contains the vertex $i$ by $Arc_{S,i}$ and the other
arc by $Arc'_{S,i}$. The following notation will be important in the sequel.
\begin{notation}\label{Not:New}
For a direction $i$, and for $0 \leq j \leq n-2k$, let $B_{i,j}$ denote the
set $S$ of $k$ consecutive edges on $V(G)$ in direction $i$, for which exactly $j$
vertices of $G$ lie in $Arc'_{S,i}$. (This implies: $|V(G) \cap Arc_{s,i}| = n-2k-j$,
and therefore restricts the parity of $j$: we have $j \equiv n (\mod 2)$ if and only if
$i$ is \textbf{not} a label of a vertex of $G$.)

\end{notation}
For example, in the left part of Figure~\ref{Fig:2New}, the set of edges depicted
in bold is $S=B_{-1,5}$, and the two corresponding arcs on the boundary of $P$ are
$Arc_{S,-1}$ (containing the vertex $-1$) and $Arc'_{S,-1}$ (containing the vertices
$7,9,11,-11,-9$). We note that, by its definition, the set $B_{i,j}$ depends also
on $n$ and $k$. In order to keep the notation concise, we prefer to leave this
dependence implicit, as the values of $n,k$ can always be understood from the context.

\subsection{An Upper Bound on the Number of Edges}
\label{sec:sub:upper}

\subsubsection{Upper bound for $q \leq n-2k$}

It is clear that all the edges in each direction are pairwise disjoint.
Hence, if a set of edges does not contain $k+1$ pairwise disjoint
edges, then it contains at most $k$ edges in each direction. As
the number of distinct directions is $n$, we get the following observation
(that was already used by Kupitz~\cite{Kupitz}):
\begin{observation}\label{Obs1}
Let $G$ be an $I_{k+1}$-free CGG of order $n$. Then $|E(G)| \leq kn$.
Moreover, if $|E(G)|=kn$, then $E(G)$
contains exactly $k$ edges in each direction.
\end{observation}

This implies the upper bound of Theorem~\ref{Thm:Main}(1).

\subsubsection{Upper bound for $n-2k < q < n-k$}

Assume now that, in addition, $G$ avoids a set $A$ of $q=n-2k+\ell$
consecutive vertices, for $1 \leq \ell < k$. In this case, the
maximal possible number of edges is reduced, since, in some of the
directions, only fewer than $k$ edges have at least one endpoint in
$A^c$. For $0 < s \leq k$, we say that direction $j$ {\it loses}
$s$ edges, if it contains exactly $k-s$ edges with at least one endpoint in $A^c$, and
denote this situation by $Loss(j)=s$. (If the number of edges in direction $j$ with
at least one endpoint in $A^c$ is greater than or equal to $k$, we say that $Loss(j)=0$.)
In the proof, we use the following fact.
\begin{claim}\label{Claim:Triv}
Let $n \in \mathbb{N}$. Then
\[
\lceil n/2 \rceil + 2 \sum_{i=1}^{n-1} \lceil i/2 \rceil = {{n+1}\choose{2}}.
\]
\end{claim}

The proof, by induction on $n$, is straightforward and is left to the reader.

\medskip



Let the allowed set $A^c$ be as described in the right part of Figure~\ref{Fig:3New},
where $\pm x$ are its rightmost vertices. (Note that $x$ is a function of $n,k,$ and $\ell$.)
We compute the number of allowed edges in each direction.

\begin{figure}[tb]
\begin{center}
\scalebox{1.0}{
\includegraphics{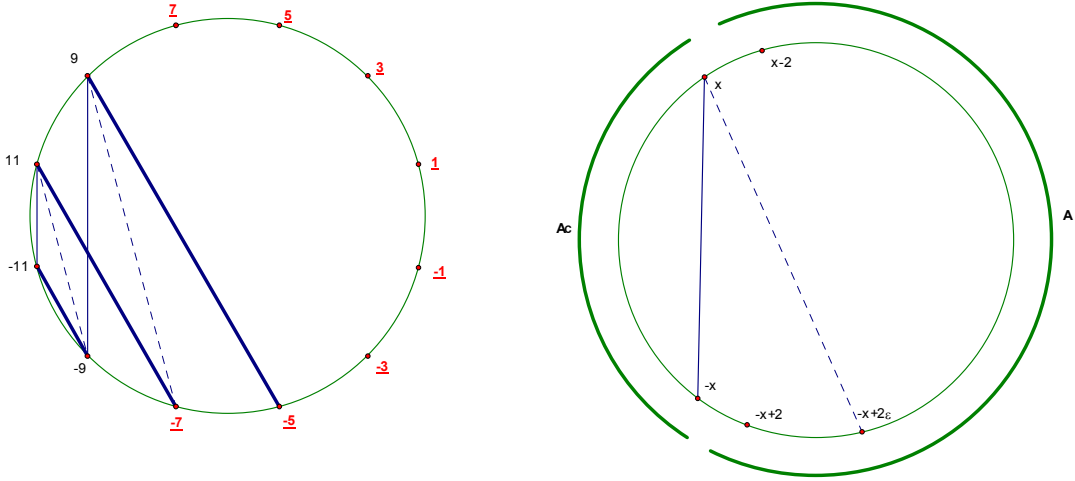}
} \caption{An illustration to the upper bound proof. The right drawing illustrates
the proof for general parameters. The left drawing contains the special case
$n=12$, $k=3$, and $\ell=2$. The avoided vertices are colored red and
underlined. The allowed edges in direction $0$ are drawn as ordinary lines, the
allowed edges in direction $1$ are drawn as punctured lines, and the allowed edges
in direction $2$ are drawn as bold lines.}
\label{Fig:3New}
\end{center}
\end{figure}

The endpoints of each edge in direction $0$ are either both allowed or both
forbidden. Hence, the number of allowed edges in direction $0$ is
$\lfloor (2k-\ell)/2 \rfloor = k- \lceil \ell/2 \rceil$, or in other words,
$Loss(0)= \lceil \ell/2 \rceil$. (Note that if $\ell$ is odd, the leftmost vertex
in $A^c$ is idle.)

In direction $\epsilon$, $0 < \epsilon \leq \ell$, the edge $e=[-x+2\epsilon,x]$ is allowed,
and so are all edges that are parallel to it and behind it (i.e., farther than $e$ from $\epsilon$).
Thus, effectively, the allowed arc is extended by $\epsilon$ vertices on its bottom side,
and hence, the number of allowed edges in this direction is $\lfloor (2k-\ell+\epsilon)/2
\rfloor$. Hence, $Loss(\epsilon) = \lceil (\ell-\epsilon)/2 \rceil$.

When $\epsilon$ reaches $\ell$, the effective length of the allowed arc becomes $2k-\ell+\ell=2k$,
and thus, the loss becomes zero. By symmetry, the situation for direction $-\epsilon$ is the
same as for direction $\epsilon$.

Therefore, the sum of the losses is
\[
\lceil \ell/2 \rceil + 2 (\lceil (\ell-1)/2 \rceil + \lceil (\ell-2)/2 \rceil + \ldots
+ \lceil 0/2 \rceil) = \ell(\ell+1)/2,
\]
where the equality holds by Claim~\ref{Claim:Triv}.

The argument above is exemplified in the left drawing of Figure~\ref{Fig:3New} in the
special case $n=12,k=3,\ell=2$. In direction $0$
there are only 2 allowed edges, which corresponds to $Loss(0)=\lceil \ell/2 \rceil =1$.
The same holds for $\epsilon=1$, and this corresponds to $Loss(1) = \lceil (\ell-\epsilon)/2
\rceil = \lceil (2-1)/2 \rceil =1$. For $\epsilon=2=\ell$, there are already 3 allowed
edges, which indeed corresponds to $Loss(2)=\lceil (\ell-2)/2 \rceil =0$.


\subsubsection{Upper and lower bounds for $q \geq n-k$}

Finally, we consider the case where $G$ avoids a set $A$ of $q
\geq n-k$ consecutive vertices. In this case, each edge of $G$
uses at least one of the vertices of $A^c$, and since $|A^c| \leq
k$, this implies that $G$ is $I_{k+1}$-free. Hence, the only
restriction on $G$ is the avoidance of $A$,
which leads to the upper bound ${{n}\choose{2}} - {{q}\choose{2}}$
on the number of edges (as there are ${{q}\choose{2}}$ edges with
both endpoints in $A$).

On the other hand, the upper bound is clearly attained by the graph
$G$ that contains all ${{n}\choose{2}} - {{q}\choose{2}}$ edges
with at least one endpoint in $A^c$, since this graph is $I_{k+1}$-free
and avoids $A$. This proves Theorem~\ref{Thm:Main}(3).

\section{Maximal CGGs Avoiding $n-2k$ Consecutive Vertices}
\label{sec:basic-construction}

In this section we consider $I_{k+1}$-free CGGs on $n$ vertices that
avoid a set $A$ of
$n-2k$ consecutive vertices. By Observation~\ref{Obs1}, the number
of edges in such a graph is at most $kn$. We show that this upper
bound is attained by a sequence of graphs $G_{n,k}$,
thus proving Theorem~\ref{Thm:Main}(1).

In Sections~\ref{sec:sub:def-Gnk}
and~\ref{sec:sub:basic-proof} we present the construction of
$G_{n,k}$ for even values of $n$ and prove that these CGGs satisfy
the conditions of Theorem~\ref{Thm:Main}(1). In Section~\ref{sec:sub:basic-odd}
we sketch the modifications in the construction of $G_{n,k}$ and in
the proof of the theorem required in the case of odd $n$.

\subsection{Definition of the Graphs $G_{n,k}$ for Even $n$}
\label{sec:sub:def-Gnk}

Let $n$ be an even integer. Define $m=(n-2k)/2$, and let the free set be
$A=\{\pm 1, \pm 3, \ldots, \pm (2m-1)\}$.

Recall that by the definition of {\it direction} given in
Section~\ref{sec:sub:notations}, directions $i$ and $i-n$ (for any
$1 \leq i \leq n$) are the same.
Hence, in this section we restrict ourselves to the set of $n$ directions
that correspond to the half-circle consisting of the consecutive vertices
$\{-m,-m+1,\ldots,0,1,\ldots,m+2k-2,m+2k-1,m+2k\}$ (where direction
$m+2k$ is equal to direction $-m$).

The CGG $G_{n,k}$ that avoids $A$ is given by the following
definition.

\begin{definition}\label{Def:Gnk}
For an even integer $n$, and $k \leq \frac{n}{2}-1$, we denote by
$G_{n,k}$ the CGG on $n$ vertices whose edge set consists of the
following sets of consecutive edges, arranged according to the directions:
\begin{itemize}
\item For $-m \leq j \leq m$, the set of edges in direction $j$ is $S_j=B_{j,|j|}$.
(See Notation~\ref{Not:New}.)

\item For $j=m+i$, $0 \leq i \leq 2k$, the set of edges in direction $m + i$ is
$S_{m+i}=B_{m+i,m-\epsilon}$, where
\[
\epsilon = \left\lbrace
  \begin{array}{c l}
    0, & \mbox{if  } 2 \mid i,  \\
    1, & \mbox{if  } 2 \nmid i.
  \end{array}
  \right.
\]
\end{itemize}
\end{definition}
\noindent Note that the sets $S_m$ and $S_{m+2k}$ are defined twice: $S_m$ appears
as $B_{m,m}$ in both clauses, while $S_{m+2k}$ appears as $S_{-m}=B_{-m,m}$ in the
first clause and as $S_{m+2k}=B_{m+2k,m}$ in the second clause (since $-m$ and $m+2k$
are the same direction).

\medskip

As the structure of $G_{n,k}$ is a bit complex, we give an
intuitive explanation of the construction before presenting the
proof. Our explanation is illustrated by two examples of
$G_{n,k}$, presented in Figure~\ref{Fig:5New}.

\begin{figure}[tb]
\begin{center}
\scalebox{1.0}{
\includegraphics{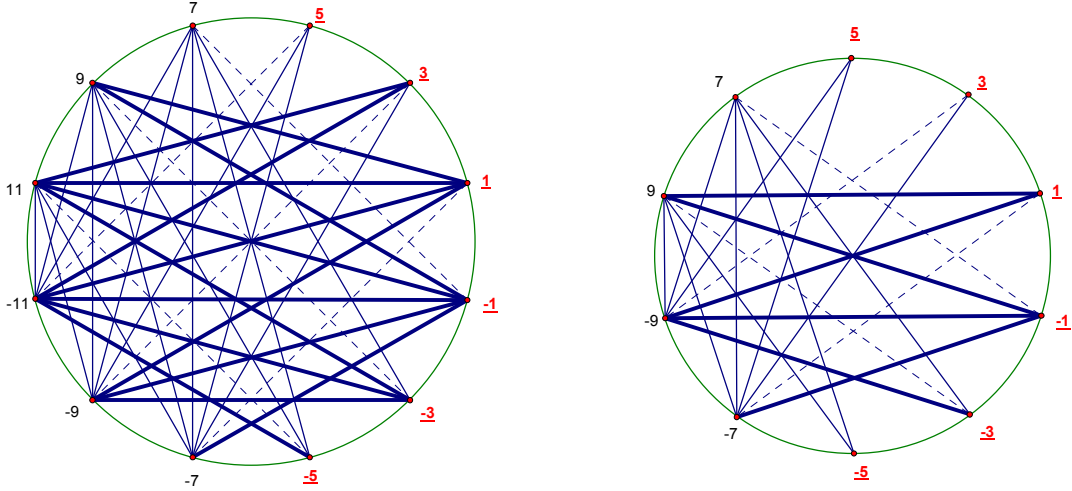}
} \caption{Examples of the CGG $G_{n,k}$. On the left, $n=12$ and $k=3$.
On the right, $n=10$ and $k=2$. In both graphs, the avoided vertices are
colored red and underlined. The edges of $S_j$ for $|j|<m$ (where $m=(n-2k)/2$)
are drawn as ordinary lines, the edges of $S_j$ for $m<j<m+2k$ are drawn as bold lines, and
the edges of $S_m,S_{m+2k}$ are drawn as punctured lines.}
\label{Fig:5New}
\end{center}
\end{figure}

\medskip

\noindent \textbf{Intuitive explanation of the construction.} We
consider the directions sequentially, in the order
$0,1,\ldots,m+2k-1,-m,-m+1,\ldots,-1,0$. (Note that direction $-m$ is the
same as direction $m+2k$, and thus, the order of directions we consider is
logical.) In each direction, we choose $k$ edges,
aiming at choosing the ``most central'' (i.e., longest possible)
edges, subject to the restriction that $G_{n,k}$ must avoid $A$.
The endpoints of the edges in
each direction $j$ form two sets of $k$ consecutive
vertices (see Figure~\ref{Fig:5New}). We call the set which is closer
to the leftmost vertex ``the left set'' and denote it by $L_j$,
and call the other set ``the right set'' and denote it by
$R_j$.\footnote{Formally, $L_j$ is the set whose distance to the
vertex labelled $n$ is smaller, where the distance $dist(S,v)$ of a set
of vertices $S$ from a single vertex $v$ is defined as
$dist(S,v)=\min_{s \in S} ord([s,v])$. A single exception is
direction $0$ for which the distances of both sets from vertex $n$ are equal.
For that direction, we choose for convenience $L_0=\{n-1,\ldots,n-(2k-1)\}$
and $R_0=\{-(n-1),\ldots,-(n-2k+1)\}$.}

In the first $m+1$ directions,
the left set $L_j=\{n-1,n-3,\ldots,n-(2k-1)\}$ is fixed, while the right set $R_j$ is moved each step
one place counterclockwise (i.e., starting with
$\{-(n-1),-(n-3),\ldots,-(n-2k+1)\}$, then
$\{-(n-3),-(n-5),\ldots,-(n-2k-1)\}$, etc.). We continue in this fashion until
$j=m$, for which $S_m=B_{m,m}$ is central (i.e., each of the arcs $Arc_{S_{m},m},
Arc'_{S_{m},m}$ contains exactly $m$ vertices). In the following directions, we try
to maintain centrality: Every second direction is central, and the directions
in between are ``nearly central'', i.e., moved half a step downward from the
central positions. This is achieved by moving $L_j$ and $R_j$ alternately:
In passing from $S_{m}$ to $S_{m+1}$, $L_m$ is moved one step counterclockwise,
while $R_m$ remains unchanged. In passing from $S_{m+1}$ to $S_{m+2}$, $R_{m+1}$
is moved one step counterclockwise, while $L_{m+1}$ is retained, etc.
This continues until we reach direction $m+2k$ (again central), which is the same
as direction $-m$. The transition from here back to direction $0$ can be viewed
as a mirror image (in reverse) of the transition from $S_0$ to $S_m$. In the last
step (from $S_{-1}$ back to $S_0$) $R_{-1}$ is moved one step counterclockwise to
$R'_{-1}$, and then the roles of $R$ and $L$ are interchanged: $L_0=R'_{-1}$ and
$R_0=L_{-1}$.

\subsection{Proof of Theorem~\ref{Thm:Main}(1) for Even $n$}
\label{sec:sub:basic-proof}

In this section we prove that $G_{n,k}$ satisfies the conditions
of Theorem~\ref{Thm:Main}(1). It is clear that $G_{n,k}$ contains
exactly $k$ edges in each direction\footnote{Note that in each direction, there are
at least $k$ allowed edges. Indeed, by the definition of $A$, the set $A^c$ contains $2k$
vertices. Hence, at least $k$ of the edges in each direction have an endpoint in $A^c$.}
and avoids the set of vertices $A=\{-2m+1,-2m+3,\ldots,2m-3,2m-1\}$. Hence,
Theorem~\ref{Thm:Main}(1) is implied by the following proposition.
\begin{proposition}\label{Prop:Basic-Main}
For each $n$ and $k$, the graph $G_{n,k}$ is $I_{k+1}$-free.
\end{proposition}

In the proof of Proposition~\ref{Prop:Basic-Main}, we use a lemma,
which requires some additional notation.
\begin{notation}\label{Not:Arc}
The endpoints of any edge $e \in E(G_{n,k})$ divide the boundary of $P$ into two
open arcs. Denote the arc that does not contain the open boundary edge
$]n-1,-(n-1)[$ by $Arc_e$, and the other arc by $Arc'_e$. We say that an edge
$e_2$ \emph{lies behind} $e_1$ (with respect to $]n-1,-(n-1)[$) if
$Arc_{e_2} \subseteq Arc_{e_1}$. (Note that an edge is said to lie behind itself.)
\end{notation}

Note that the notations $Arc_e$ and $Arc'_e$ differ from the
notations $Arc_{S_j,j}$ and $Arc'_{S_j,j}$ defined in
Section~\ref{sec:sub:notations}. This difference is intentional, and
both types of notation will be used in the sequel.

\begin{lemma}\label{Lemma:Basic-Main}
For any $e \in E(G_{n,k})$, the open arc $Arc_e$ contains at least
$m-1$ vertices of $G$. Furthermore, if $e$ emanates from a positive-labelled
vertex, then $Arc_e$ contains at least $m$ vertices of $G$.
\end{lemma}

\begin{proof}
We consider separately the sets of edges in $S_j$ for $-m \leq j \leq m$ and for
$m+1 \leq j \leq m+2k-1$.

\medskip

\textbf{Case I: $e \in S_j$ for $-m \leq j \leq m$.} By Definition~\ref{Def:Gnk},
we have $S_j=B_{j,|j|}$. Recall that, as explained in Section~\ref{sec:sub:notations},
$S_j=B_{j,|j|}$ divides the remaining vertices on the boundary of
$\mathrm{conv} (V(G))$ into two arcs: $Arc'_{S_j,j}$ (which includes the open edge
$]n-1,-(n-1)[$) and $Arc_{S_j,j}$. As by the definition of
$B_{j,|j|}$, $Arc'_{S_j,j}$ contains $|j| \leq m$ vertices of $G$, it
follows that $Arc_{S_j,j}$ contains at least $n-2k-|j| \geq m$
vertices of $G$. Since for each $e \in S_j$, the arc $Arc_e$
includes the arc $Arc_{S_j,j}$, it follows that $Arc_e$ contains
at least $m$ vertices of $G$.

\medskip

\textbf{Case II: $e \in S_j$ for $m+1 \leq j \leq m+2k-1$.} By
Definition~\ref{Def:Gnk}, $S_j = S_{m+i}=B_{m+i,m-1}$ for
odd $i$ and $S_j= S_{m+i}=B_{m+i,m}$ for even $i$. For $S_{m+i}=B_{m+i,m}$,
each of the arcs $Arc_{S_{m+i},m+i}$ and $Arc'_{S_{m+i},m+i}$ contains
exactly $m$ vertices of $G$, and thus, for any $e \in B_{m+i,m}$, the
arc $Arc_e$ contains at least $m$ vertices of $G$. Similarly, for
$S_{m+i}=B_{m+i,m-1}$, the arcs $Arc_{S_{m+i},m+i}$ and
$Arc'_{S_{m+i},m+i}$ contain $m+1$ and $m-1$ vertices of $G$
respectively, and thus, for any  $e \in B_{m+i,m-1}$, the arc $Arc_e$
contains at least $m-1$ vertices of $G$.

When trying to show that, for some of the edges, the arc $Arc_e$ contains at
least $m$ vertices of $G$, we face a difficulty: Unlike the case $-m \leq j \leq m$,
where for any $e \in S_j$ the arc $Arc_e$ includes the arc $Arc_{S_j,j}$, here
there is a variance between the edges. For some of the edges $e$,
$Arc_e$ includes $Arc_{S_{m+i},m+i}$ (which contains $m+1$ vertices of
$G$), while for the other edges, $Arc_e$ includes $Arc'_{S_{m+i},m+i}$ (which
contains only $m-1$ vertices of $G$).

However, we observe that the edges $e$ for which $Arc_e$ includes
$Arc_{S_{m+i},m+i}$ are exactly those which emanate from
positive-labelled vertices (see Figure~\ref{Fig:7New}). Indeed, if $e \in S_{m+i}$
emanates from a positive-labelled vertex, then the vertex $m+i$ and the open edge
$]-(n-1),n-1[$ lie on opposite sides of the edge $e$. Thus, by the definitions
of $Arc_e$ and of $Arc_{S_{m+i},m+i}$, it follows that $Arc_{S_{m+i},m+i} \subset
Arc_{e}$. On the other hand, if $e$ emanates from a negative-labelled vertex,
then $m+i$ and $]-(n-1),n-1[$ lie on the same side of $e$, and thus,
$Arc'_{S_{m+i},m+i}$ is included in $Arc_{e}$.
Hence, we conclude that for any $e \in E(G_{n,k})$ that emanates from a
positive-labelled vertex, $Arc_e$ contains at least $m$ vertices of $G$. This
completes the proof of the lemma.
\end{proof}

\begin{figure}[tb]
\begin{center}
\scalebox{1.0}{
\includegraphics{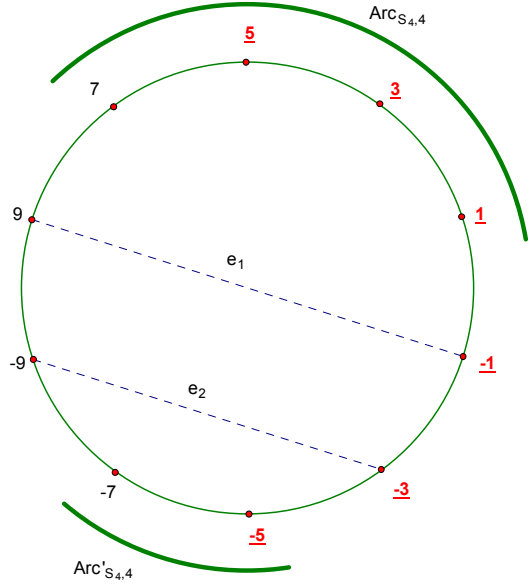}
}
\caption{An illustration for the proof of Lemma~\ref{Lemma:Basic-Main},
corresponding to $n=10$ and $k=2$. The avoided vertices are colored red and underlined.
For the edge $e_1= [9,-1]$, the arc $Arc_{e_1}$ includes $Arc_{S_4,4}$, while
for the edge $e_2=[-9,-3]$, the arc $Arc_{e_2}$ does not include $Arc_{S_4,4}$.}
\label{Fig:7New}
\end{center}
\end{figure}

\begin{proof}(of Proposition~\ref{Prop:Basic-Main})

Assume, on the contrary, that $G_{n,k}$ contains a set $S$
of $k+1$ pairwise disjoint edges. By the construction of $G_{n,k}$, each
edge of $S$ has at least one endpoint in the set
\[
A^c=\{\pm (n-1), \pm (n-3), \ldots, \pm (n-2k+1)\},
\]
that consists of $k$ positive-labelled vertices and $k$
negative-labelled vertices (see Figure~\ref{Fig:7.5}).
As the edges of $S$ are disjoint, and there are $k+1$ edges in $S$, there exist
two edges $[i_1,t_1],[i_2,t_2] \in S$, such that:
\begin{itemize}
\item $[i_1,t_1]$ does not have a negative-labelled endpoint in $A^c$. In the
notations of Figure~\ref{Fig:7.5}, $[i_1,t_1]$ does not have an endpoint in
$D_3$, so it must have an endpoint in $D_2$. Moreover, $Arc_{[i_1,t_1]}
\subset D_1 \cup D_2$.

\item $[i_2,t_2]$ does not have a positive-labelled endpoint in $A^c$.
In the notations of Figure~\ref{Fig:7.5}, $[i_2,t_2]$ does not have an
endpoint in $D_2$, so it must have an endpoint in $D_3$. Moreover,
$Arc_{[i_2,t_2]} \subset D_1 \cup D_3$.
\end{itemize}

It follows that none of the arcs $Arc_{[i_1,t_1]},Arc_{[i_2,t_2]}$ is
included in the other, and therefore they are disjoint (see
Figure~\ref{Fig:7.5}).

\begin{figure}[tb]
\begin{center}
\scalebox{1.0}{
\includegraphics{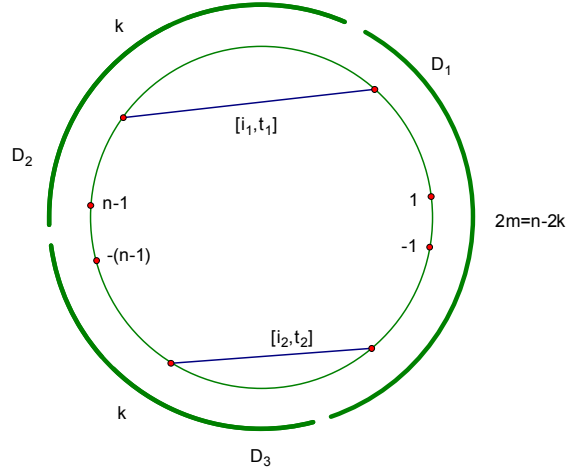}
}
\caption{An illustration for the proof of Proposition~\ref{Prop:Basic-Main}.
The avoided set $A$ consists of the arc $D_1$ (containing $2m=n-2k$
vertices). The set $A^c$ consists
of two arcs: The arc $D_2$ of $k$ positive-labelled vertices,
and the arc $D_3$ of $k$ negative-labelled vertices. The edges $[i_1,t_1],[i_2,t_2]$
are chosen such that $[i_1,t_1]$ does not have endpoints in $D_3$, and
$[i_2,t_2]$ does not have endpoints in $D_2$.}
\label{Fig:7.5}
\end{center}
\end{figure}


We would like to show that the arc $Arc_{[i_1,t_1]}$ contains at least $m$
vertices of $G$ that are not used by any edge of $S$, and the arc $Arc_{[i_2,t_2]}$
contains at least $m-1$ such vertices. This will be a contradiction to the assumption
$|S|=k+1$, since by the assumption, $S$ uses $2k+2$ vertices of $G$, and thus, only
$n-(2k+2)=2m-2$ vertices of $G$ are not used by $S$.

We say that an edge $e \in S$ is {\it extremal} if no other edge of $S$ lies behind $e$
(see Notation~\ref{Not:Arc}).
It is immediate that behind any edge of $S$ there is an extremal one.
Note that since
the edges of $S$ are pairwise disjoint, if $e'$ is an extremal edge of $S$ then the
vertices of $Arc_{e'}$ are not used by any edge of $S$.

Let $e_1$ and $e_2$ be extremal edges that lie behind $[i_1,t_1]$ and $[i_2,t_2]$,
respectively. (Note that we may have $e_1=[i_1,t_1]$ and/or
$e_2=[i_2,t_2]$.) By Lemma~\ref{Lemma:Basic-Main}, each of the arcs $Arc_{e_1}$ and $Arc_{e_2}$
contains at least $m-1$ vertices of $G$, and these vertices are certainly not used by any
edge of $S$. Moreover, since both endpoints of $e_1$ are contained in
$Arc_{[i_1,t_1]} \subset \{n-1,n-3,\ldots,1,-1,\ldots,-(n-2k-3),-(n-2k-1)\}$, it follows
that $e_1$ emanates from a positive-labelled vertex. (This holds since $e_1$ must emanate
from a vertex in $A^c$, and $Arc_{[i_1,t_1]} \cap A^c$ contains only positive-labelled vertices.)
Hence, Lemma~\ref{Lemma:Basic-Main} actually implies that $Arc_{e_1}$ contains at least $m$
vertices of $G$ that cannot be used by edges of $S$. In total, we have at least $m+(m-1)=2m-1$
vertices of $G$ that cannot be used by edges of $S$ (and these vertices are distinct
since $Arc_{e_1} \cap Arc_{e_2} \subset Arc_{[i_1,t_1]} \cap Arc_{[i_2,t_2]} = \emptyset$).
This leads to a contradiction as explained above, and thus concludes the proof of
Proposition~\ref{Prop:Basic-Main}.
\end{proof}

\subsection{Maximal CGGs Avoiding $n-2k$ Consecutive Vertices for Odd $n$}
\label{sec:sub:basic-odd}

The case of odd $n$ is very similar to the case of even $n$, and is, in some respect, even
easier. Thus, we only
sketch briefly the required modifications in the notations, in the construction
of $G_{n,k}$, and in the proof of Theorem~\ref{Thm:Main}(1).

\medskip

\textbf{Notations.} The first slight change is in the notations.
Let $n$ be odd, and let \textbf{$m=\lfloor (n-2k)/2 \rfloor = (n-2k-1)/2$}. As $|A|=2m+1$ is odd,
we take the vertices of $V(G)$ to be the even-labelled vertices of $P$,
as shown in Figure~\ref{Fig:3Old}. This is in contrast with the case of even $n$,
where $|A|$ is even, and the vertices of $V(G)$ are taken to be the odd-labelled
vertices of $P$ (see Section~\ref{sec:sub:notations}).

\begin{figure}[tb]
\begin{center}
\scalebox{1.0}{
\includegraphics{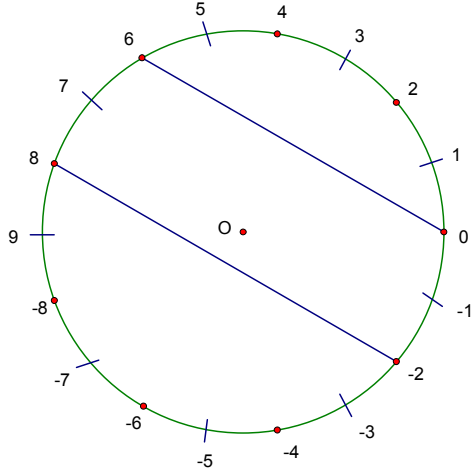}
}
\caption{Representation of a CGG $G$ on $9$ vertices using a regular polygon $P$
on $18$ vertices. The vertices of $G$ are the even-labelled vertices of $P$.
The set of edges shown is $B_{3,3}$.}
\label{Fig:3Old}
\end{center}
\end{figure}

\medskip

\textbf{Construction.} The second change is in the definition of $G_{n,k}$.
While the definition is very similar to the even case, there are a few changes.
The formal definition of $G_{n,k}$ in the odd case is as follows:
\begin{definition}\label{Def:Gnk-odd}
For an odd integer $n$, and $k \leq \lfloor \frac{n}{2} \rfloor-1$, we denote by
$G_{n,k}$ the CGG on $n$ vertices whose edge set is the union of the
following sets of consecutive edges, arranged according to the directions:
\begin{itemize}
\item For $-(m+1) \leq j \leq m+1$, the set of edges in direction $j$ is $S_j=B_{j,|j|}$.

\item For $j=m+i$, $0 \leq i \leq 2k+1$, the set of edges in direction $m + i$ is
$S_{m+i}=B_{m+i,m+\epsilon}$, where
\[
\epsilon = \left\lbrace
  \begin{array}{c l}
    0, & 2 \mid i,  \\
    1, & 2 \nmid i.
  \end{array}
  \right.
\]
\end{itemize}
\end{definition}
(Note that the edges in directions $\pm m, \pm (m+1)$ are defined twice. Of course,
the definitions coincide.) An example of $G_{n,k}$, with $n=13$ and $k=m=3$,
is presented in Figure~\ref{Fig:6New}.

\begin{figure}[tb]
\begin{center}
\scalebox{1.0}{
\includegraphics{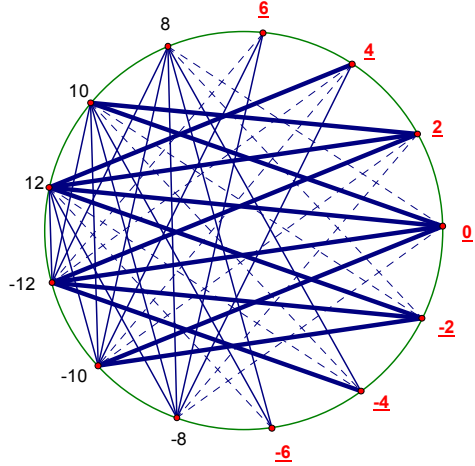}
} \caption{An example of the CGG $G_{n,k}$, for $n=13$ and $k=3$.
The avoided vertices are
colored red and underlined. The sets $S_j$ for $|j|<m$ are drawn as
ordinary lines, the sets $S_j$ for $m+1<j<m+2k$ are drawn as bold lines, and
the sets $S_m,S_{m+1},S_{-m},S_{-m-1}$ are drawn as punctured lines.}
\label{Fig:6New}
\end{center}
\end{figure}

\medskip

\textbf{Main Lemma.} The main lemma used in the proof of Theorem~\ref{Thm:Main}(1)
in the even case, Lemma~\ref{Lemma:Basic-Main}, is replaced by the following
lemma, which turns out to be even stronger:
\begin{lemma}\label{Lemma:Basic-Main-Odd}
For any $e \in E(G_{n,k})$, the open arc $Arc_e$ contains at least
$m$ vertices of $G$.
\end{lemma}
To prove Lemma~\ref{Lemma:Basic-Main-Odd}, we note that in the odd case, a
set of the form $B_{j,m}$ or $B_{j,m+1}$ leaves $m$ and $m+1$ vertices of $G$ on its
sides (compared to $m$ and $m$ or $m-1$ and $m+1$ in the even case). As a result, since
for any $j$ such that $m \leq j \leq m+2k+1$, $S_j$ is of one of the
forms $B_{j,m}$ or $B_{j,m+1}$, it follows immediately that for any
edge $e \in S_j$ (for $m \leq j \leq m+2k+1$), $Arc_e$ contains at least
$m$ vertices of $G$. The proof for $-m < j < m$ is identical
to the proof of the corresponding statement of Lemma~\ref{Lemma:Basic-Main}
in the even case.

\medskip

\textbf{Proof of Theorem~\ref{Thm:Main}(2).} Due to the stronger form of
Lemma~\ref{Lemma:Basic-Main-Odd}, the proof is even simpler than in the even
case. Same as in the even case, we assume on the contrary that $G_{n,k}$ contains
a set $S$ of $k+1$ pairwise disjoint edges, and show that $S$ contains two
edges $[i_1,t_1],[i_2,t_2]$ such that $[i_1,t_1]$ does not have a negative-labelled
endpoint in $A^c$, and $[i_2,t_2]$ does not have a positive-labelled endpoint in
$A^c$. Furthermore, each of the closed arcs $Arc_{[i_1,t_1]},Arc_{[i_2,t_2]}$ contains
the endpoints of an extremal edge (which may be one of the edges $[i_1,t_1],[i_2,t_2]$). We
denote these extremal edges by $e_1$ and $e_2$, respectively. By
Lemma~\ref{Lemma:Basic-Main-Odd},
each of the arcs $Arc_{e_1},Arc_{e_2}$ contains at least $m$ vertices
of $G$ that are not used by edges of $S$. Hence, the number of vertices that
are used by edges of $S$ is at most $n-2m=2k+1$, contradicting the assumption
that $S$ consists of $k+1$ {\it disjoint} edges. This completes the proof of
the theorem for odd $n$.

\section{CGGs Avoiding $n-2k+ \ell$ Consecutive Vertices}
\label{sec:extension}

In this section we consider $I_{k+1}$-free CGGs on $n$ vertices that
admit a free boundary arc
of order $n-2k+\ell$, for $1 \leq \ell <k$. As shown
in Section~\ref{sec:notations}, the number of edges in such a
graph is at most $kn - \ell(\ell+1)/2$. We show that this upper
bound is attained by a sequence of graphs
$G_{n,k,\ell}$, thus proving Theorem~\ref{Thm:Main}(2).

Same as in Section~\ref{sec:basic-construction}, we write $m=\lfloor (n-2k)/2 \rfloor
(=\lfloor n/2 \rfloor -k$),
so that $|A|=2m+\ell$ if $n$ is even, and $|A|=2m+\ell+1$ if $n$ is odd.
As indicated in Section~\ref{sec:sub:notations}, we take $V(G)$ to be the set
of odd-labelled vertices of $P$ if $|A|$ is even, or the set of the even-labelled
vertices of $P$ if $|A|$ is odd. We define $K=A^c$, and divide $K$
into three blocks: $K_+$ (the uppermost $k-\ell$ vertices of $K$), $K_{-}$ (the lowest
$k-\ell$ vertices of $K$), and $K_0$ (the $\ell$ vertices in the middle), see
Figure~\ref{Fig:9NewNew}.

\begin{figure}[tb]
\begin{center}
\scalebox{1.0}{
\includegraphics{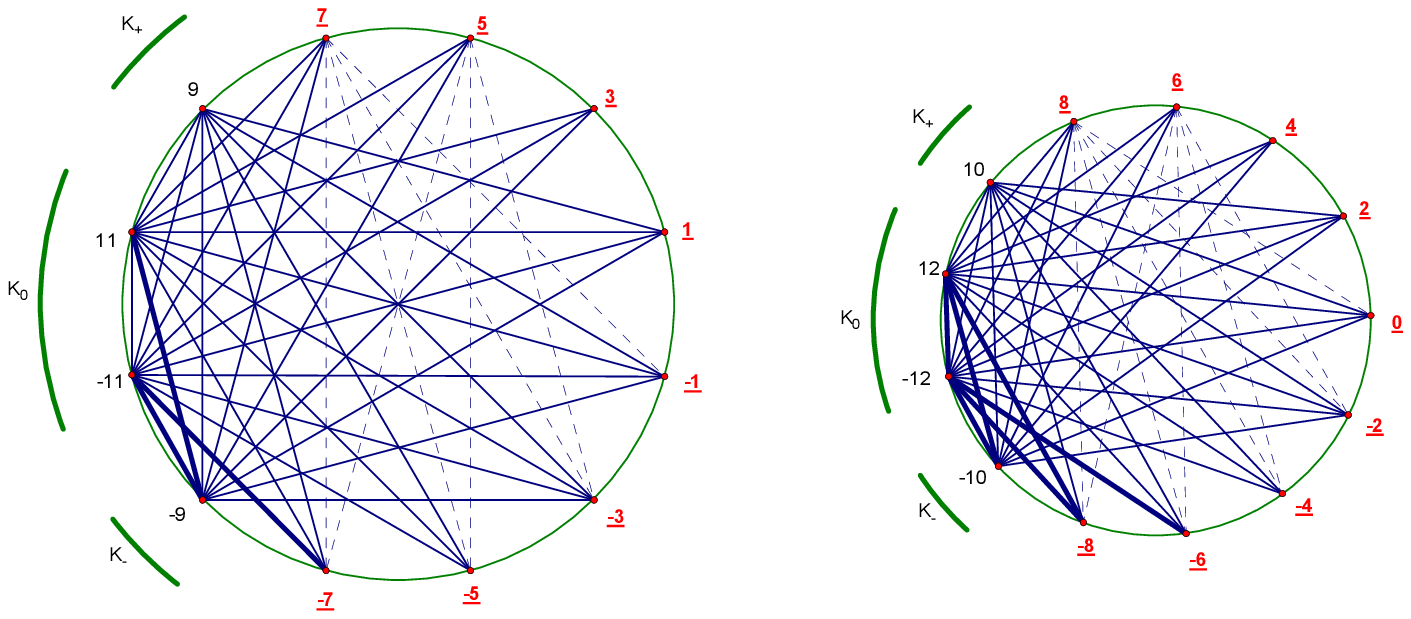}
}
\caption{The CGG $G_{n,k,\ell}$ for different values of $n,k,$ and $\ell$.
On the left: $n=12,k=3$, and $\ell=2$. On the right: $n=13,k=3$, and $\ell=2$.
In both figures, the avoided vertices are colored red and underlined. The
edges of $G_{n,k,\ell}$ that do not belong to $G_{n,k}$ are denoted by bold lines.
The edges of $G_{n,k,\ell}$ that belong to $G_{n,k}$ are denoted by ordinary lines.
The edges of $G_{n,k}$ that do not appear in $G_{n,k,\ell}$ are denoted by punctured lines.
The sets $K_+,K_-,$ and $K_0$ are marked in both graphs.}
\label{Fig:9NewNew}
\end{center}
\end{figure}

For ease of exposition, we consider the cases of even and odd $n$ separately.
In Sections~\ref{sec:sub:def-Gnkl} and~\ref{sec:sub:advanced-proof} we present
the construction of $G_{n,k,\ell}$ for even values of $n$ and prove that these CGGs satisfy
the conditions of Theorem~\ref{Thm:Main}(2). In Section~\ref{sec:sub:advanced-odd}
we sketch the modifications in the construction of $G_{n,k,\ell}$ and in
the proof of the theorem required in the case of odd $n$.

\subsection{Construction of the Graphs $G_{n,k,\ell}$ for Even $n$}
\label{sec:sub:def-Gnkl}

For $n=2k+2m$ even, we describe the edges of $G_{n,k,\ell}$ by direction, as follows.
The directions range over the half-circle from $-\ell-m$ up to $2k+m-\ell (=-\ell-m+n)$.
The set of edges of $G_{n,k,\ell}$ in direction $j$ is denoted by $S_j$.
We divide the directions into five sub-ranges:
\begin{enumerate}
\item $-\ell-m \leq j \leq -\ell$,
\item $-\ell \leq j \leq 0$,
\item $0 \leq j \leq \ell$,
\item $\ell \leq j \leq \ell+m$,
\item $\ell+m \leq j \leq 2k+m-\ell$.
\end{enumerate}

The edges in the subranges are defined as follows:

\begin{itemize}
\item \textbf{$j=0$:} For $j=0$, $S_0$ consists of the
$k-\lceil \ell/2 \rceil$ leftmost vertical edges.

\item \textbf{$0 < j \leq \ell$:} We enlarge $K$ by adding $j$ vertices at the bottom.
$S_j$ consists of the $k-\lceil (\ell-j)/2 \rceil$ parallel edges that fit into this
enlarged arc, with one free vertex in the middle if $\ell -j$ is odd.

\item \textbf{$-\ell \leq j < 0$:} Same as above, with ``bottom'' replaced by ``top''
and $j$ replaced by $|j|$.

\item \textbf{$\ell \leq j \leq \ell+m$:} We enlarge $K$ by adding $j$ vertices
at the bottom. $S_j$ consists of the $k$ most central parallel edges that fit into this
enlarged arc, leaving $j-\ell$ free vertices in its center. ($S_{\ell+m}$ is already
central.)

\item \textbf{$-\ell-m \leq j \leq -\ell$:} Same as above, with ``bottom'' replaced by
``top'', and $j$ by $|j|$.

\item \textbf{$\ell+m \leq j < 2k+m-\ell$:} $S_j$ is a $k$-block, central if
$j \equiv \ell+m (\bmod 2)$, and near central if $j \equiv \ell+m+1 (\bmod 2)$. In the
latter case, there are $m+1$ unused vertices of $G$ above $S_j$, and $m-1$ below.
\end{itemize}

The idea behind the construction is the following: For each direction $j$, we consider
the set of allowed edges (i.e., edges with at least one endpoint included in $A^c$).
The edges of this set are consecutive. If their number is at most $k$, we take all of
them to $G_{n,k,\ell}$. If their number exceeds $k$, we take the $k$ edges which are the
closest to the center, and if there are two $k$-blocks with the same distance from the
center, we take the lower one among them.

Another way to view the construction is to compare $G_{n,k,\ell}$ to the corresponding
graph $G_{n,k}$. When moving from $G_{n,k}$ to $G_{n,k,\ell}$, $\ell$ vertices are added
to the forbidden set $A$. As a result, we lose part of the edges that emanate from these
$\ell$ vertices. When defining $G_{n,k,\ell}$, we retain all edges of $G_{n,k}$ that
remain allowed, and add some edges to compensate for the edges
that become forbidden. Specifically, in each direction $j$ in which all edges of $G_{n,k}$ remain
allowed, we take $S_j$ to be the set of edges of $G_{n,k}$ in that direction.
In the directions in which some of the edges
of $G_{n,k}$ become forbidden, we observe that these edges are those
emanating from a set of consecutive vertices. For
each such direction, we remove from the edges of $G_{n,k}$ the edges that become forbidden
and compensate for them by adding a set of consecutive edges on the other
side of the set of $G_{n,k}$. If the number of edges we can add is smaller than the
number of edges that become forbidden (since the direction ``ends'', in a
single vertex or in a boundary edge), we add the maximal possible number
of edges.

As a result, $G_{n,k,\ell}$ satisfies the following, for each $j$:
\begin{itemize}
\item The set of edges of $G_{n,k,\ell}$ in direction
$j$ is consecutive.

\item If at most $k$ edges in direction $j$ are allowed, then $G_{n,k,\ell}$
contains all allowed edges in direction $j$. Otherwise, $G_{n,k,\ell}$
contains $k$ edges in direction $j$.
\end{itemize}

An example of the construction $G_{n,k,\ell}$, for $n=12$, $k=3$,
and $\ell=2$, is presented in the left part of Figure~\ref{Fig:9NewNew}.
In the figure, in directions $-5,-4,-3,-2,4,5,6$ all edges remain allowed,
in directions $-1,2,3$ a single edge becomes forbidden, and in directions $0,1$
two edges become forbidden. On the other hand, a single compensating edge is
added to each of directions $0,1,2,3$. As a result, $G_{12,3,2}$ contains $2$
edges in directions $-1,0,1$, and $3$ edges in each other direction.

\subsection{Proof of Theorem~\ref{Thm:Main}(2)}
\label{sec:sub:advanced-proof}

In this section we present the proof of Theorem~\ref{Thm:Main}(2) in the case
of even $n$.

Assume $n=2m+2k$, $m,k \geq 2$, and $1 \leq \ell < k$. We have to establish
three claims about the graph $G_{n,k,\ell}$ constructed above:
\begin{enumerate}
\item $G_{n,k,\ell}$ is $I_{k+1}$-free.

\item $G_{n,k,\ell}$ has $nk-{{\ell+1}\choose{2}}$ edges.

\item $G_{n,k,\ell}$ has a free boundary arc of order $2m+\ell$.
\end{enumerate}

\textbf{Item~(3)} is obvious, since, by our construction, $A$ is a free arc of $G_{n,k,\ell}$.

\textbf{Item~(2):} The number of edges of $G_{n,k,\ell}$ in direction $j$ is:
$k-\lceil \ell/2 \rceil$ for $j=0$, $k-\lceil (\ell-|j|)/2 \rceil$ for $1 \leq |j| \leq \ell$,
and $k$ for all other $n-(2\ell-1)$ directions. It follows that
\[
e(G_{n,k,\ell}) = nk-2 \sum_{i=1}^{\ell-1} \lceil i/2 \rceil - \lceil \ell/2 \rceil =
nk - \ell(\ell+1)/2,
\]
by Claim~\ref{Claim:Triv}.

\textbf{Item~(1):} This is the main claim. Assume, on the contrary, that $S$ is a set
of $k+1$ pairwise disjoint edges of $G_{n,k,\ell}$. Since $|K_0 \cup K_-| = k$, there
is an edge $e^{+} \in S$ that does not use any vertex of $K_0 \cup K_-$. It follows that
$e^{+}$ is not vertical, that the left vertex of $e^+$ belongs to $K_+$, and that its
right vertex belongs to $K_-$ or to $A$. Among all edges in $S$ whose left endpoint
is in $K_+$, choose the one whose left endpoint is as far to the right as possible
and call it $e_+$. By the construction of $G_{n,k,\ell}$, there are at least $m$
vertices above $e_+$ that are not used by any edge in $S$.

Similarly, there is an edge $e^- \in S$ that does not use any vertex in $K_0 \cup K_+$.
Choose such an edge $e_-$ whose left endpoint is as far to the right as possible. By the
construction of $G_{n,k,\ell}$, there are at least $m-1$ vertices below $e_-$ that are not
used by any edge of $S$. It follows that $S$ uses altogether at most
$n-m-(m-1)=2m+2k-(2m-1)=2k+1$ vertices, a contradiction.



\subsection{Maximal CGGs Avoiding $n-2k+ \ell$ Consecutive Vertices for Odd $n$}
\label{sec:sub:advanced-odd}

The construction of $G_{n,k,\ell}$ in the odd case is almost exactly
the same as in the even case. The only difference in the definition
is in the last sub-range of directions, namely, $\ell +m \leq j< 2k+m-\ell$. In the odd case,
the directions $2k+m-\ell,2k+m-\ell+1$ are added to that sub-range, and (unlike the even case),
$S_j$ is always an almost central $k$-block, such that there are $m$ unused vertices
of $G$ above it if $j \not \equiv \ell+m (\bmod 2)$, and $m+1$ unused vertices
above it if $j \equiv \ell+m (\bmod 2)$.

An example of $G_{n,k,\ell}$, with $n=13,k=3,$ and $\ell=2$, is
presented in the right part of Figure~\ref{Fig:9NewNew}.
In the figure, in directions $-5,-4,-3,-2,5,6,7$ all edges remain allowed,
in directions $-1,3,4$ a single edge becomes forbidden, and in directions $0,1,2$
two edges become forbidden. On the other hand, a single compensating edge is
added to each of directions $0,1,3,4$, and two compensating edges are added
to direction $2$. As a result, $G_{13,3,2}$ contains $2$
edges in directions $-1,0,1$, and $3$ edges in each other direction.

The proof of Theorem~\ref{Thm:Main}(2) in the odd case is almost identical to
the proof in the even case. The only difference is that in the odd case, both
$e_+$ and $e_-$ leave at least $m$ unused vertices of $G$ behind them. This
implies that at most $n-2m=2k+1$ vertices of $G$ are used
by edges of $B$, a contradiction to the assumption that the $k+1$ edges in $B$
are disjoint. This completes the proof of Theorem~\ref{Thm:Main}(2) for odd $n$.


\begin{thebibliography}{99}

\bibitem{AFPS14} E. Ackerman, J. Fox, J. Pach, and A. Suk, On Grids in Topological Graphs,
{\it Comput. Geom.} \textbf{47(7)} (2014), pp.~710--723.

\bibitem{ACFFHHHW10} O. Aichholzer, S. Cabello, R. Fabila-Monroy,
D. Flores-Pe˜\~{n}aloza, T. Hackl, C. Huemer, F. Hurtado, and D. R. Wood,
Edge-Removal and Non-Crossing Configurations in Geometric Graphs,
{\it Disc. Math. Theor. Comput. Sci.,} \textbf{12(1)} (2010), pp.~75--86.


\bibitem{AA89} N. Alon and P. Erd\H{o}s, Disjoint Edges in Geometric Graphs,
{\it Disc. Comput. Geom.} \textbf{4} (1989), pp.~287-–290.

\bibitem{BKV} P. Brass, G. K\'{a}rolyi, and P. Valtr,
A Tur\'{a}n-type Extremal Theory of Convex Geometric Graphs, in:
Discrete and Computational Geometry -- the Goodman-Pollack
Festschrift, B. Aronov et al. eds, Springer-Verlag, 2003,
pp.~275--300.

\bibitem{BMP} P. Brass, W. Moser, and J. Pach, Research Problems
in Discrete Geometry, Springer-Verlag, 2005.

\bibitem{CP92} V. Capoyleas and J. Pach, A Tur\'{a}n-Type Theorem on Chords
of a Convex Polygon, {\it J. Combin. Th. Ser. B} \textbf{56} (1992), pp.~9--15.

\bibitem{Cerny05} J. \v{C}ern\'{y}, Geometric Graphs with no Three Disjoint Edges,
{\it Disc. Comput. Geom.} \textbf{34(4)} (2005), pp.~679–-695.

\bibitem{DKM02} A. Dress, J. H. Koolen, and V. Moulton, On Line Arrangements in the
Hyperbolic Plane, {\it Europ. J. Combin.} \textbf{23} (2002), pp.~549-–557.

\bibitem{Erdos46} P. Erd\H{o}s, On Sets of Distances of $n$ Points,
{\it Amer. Math. Monthly} \textbf{53} (1946), pp.~248-–250.

\bibitem{Felsner} S. Felsner, Geometric Graphs and Arrangements,
Vieweg Verlag, 2004.

\bibitem{FPS13} J. Fox, J. Pach, and A. Suk, The Number of Edges in
$k$-Quasi-Planar Graphs, {\it SIAM J. Disc. Math.} \textbf{27(1)} (2013),
pp.~550--561.

\bibitem{Caterpillar1} F. Harary and A.J. Schwenk, The Number of
Caterpillars, {\it Discrete Math.} \textbf{6} (1973), pp.~359--365.

\bibitem{Jonsson05} J. Jonsson, Generalized Triangulations and Diagonal-Free
Subsets of Stack Polyominoes, {\it J. Combin. Th. Ser. A} \textbf{112} (2005),
pp.~117-–142.

\bibitem{Keller-Perles} C.~Keller and M. A.~Perles, On the
Smallest Sets Blocking Simple Perfect Matchings in a Convex
Geometric Graph, {\it Israel J. Math.} \textbf{187} (2012),
pp.~465--484.

\bibitem{Kupitz} Y.S. Kupitz, On Pairs of Disjoint Segments in
Convex Position in the Plane, Annals Discrete Math., \textbf{20}
(1984), pp.~203--208.

\bibitem{Pach15} J. Pach, Geometric Intersection Patterns and the Theory of
Topological Graphs, preprint, 2015. Available online at:
http://www.cims.nyu.edu/~pach/publications/PachICM032314.pdf.


\bibitem{TV99} G. T\'{o}th and P. Valtr, Geometric Graphs with Few Disjoint
Edges, {\it Disc. Comput. Geom.} \textbf{22(4)}, pp.~633–-642, 1999.

\bibitem{Woodall} D.R. Woodall, Thrackles and Deadlock, in: Combinatorial
Mathematics and its Applications, D.J.A. Welsh, ed., Academic
Press, London, 1971, pp.~335--347.

\end{thebibliography}
\end{document}